\documentclass[a4paper,12pt]{article}

\usepackage[english]{babel}

\usepackage[a4paper,top=2.5cm,bottom=2.5cm,left=3cm,right=3cm,marginparwidth=1.75cm]{geometry}

\usepackage{amsmath, amssymb, amsthm}
\usepackage{amscd}
\usepackage{extarrows}

\usepackage{float}
\usepackage{graphicx}
\usepackage{xcolor}

\usepackage[colorlinks=true, allcolors=blue]{hyperref}

\newtheorem*{claim}{Claim}

\newtheorem{lemma}{Lemma}[section]
\newtheorem{proposition}{Proposition}[section]
\newtheorem{definition}{Definition}[section]
\newtheorem{example}{Example}
\newtheorem{operation}{Operation}

\newtheorem*{maintheorem}{Main Theorem}
\newtheorem*{remark}{Remark}
\newtheorem*{note}{Note}

\title{Building Foliations from Heegaard Diagrams}
\author{Shangjun Shi\thanks{School of Mathematical Sciences, 
  East China Normal University, Shanghai 200241, China. \texttt{52280155010@stu.ecnu.edu.cn}}
  \and 
  Yanqing Zou\thanks{School of Mathematical Sciences, 
  Key Laboratory of MEA (Ministry of Education) \& 
  Shanghai Key Laboratory of PMMP, 
  East China Normal University, Shanghai 200241, China.
  \texttt{yqzou@math.ecnu.edu.cn}}}

\begin{document}
\maketitle

\begin{abstract}
Every closed orientable $3$-manifold admits both a Heegaard splitting and a coorientable codimension-one foliation.  We address the \emph{foliation realization problem}: constructing a foliation directly from a Heegaard diagram of arbitrary genus.

Using Gabai's sutured manifold theory, we introduce the notion of a \emph{meridional sutured handlebody} and prove that every such handlebody decomposes, via basic sutured decomposition operations, into a Reeb 
component together with product disks.  

We then present a three-step construction: (i)~building two meridional sutured handlebodies from a given Heegaard diagram, (ii)~gluing them along compatible disk regions by reversing disk decompositions, and 
(iii)~filling the remaining cavities with a meridional sutured handlebody and some Reeb components.  The construction applies to Heegaard diagrams of any genus.
\end{abstract}

\section{Introduction}

Every closed orientable $3$-manifold admits a Heegaard splitting~\cite{Bing1959,Heegaard1898,Moise1952}, a coorientable codimension-one foliation~\cite{Lickorish1965, Novikov1965, Thurston1976}, and an open book decomposition~\cite{Alexander1923}.  A long-standing theme in geometric topology is the direct construction of one structure from another. Two directions are classical: an open book decomposition induces both a foliation (by extending the mapping torus with Reeb components along the binding) and a Heegaard splitting (by splitting the pages or thickening the binding neighbourhood).

The obstruction is structural. Casson and Gordon~\cite{CassonGordon1987} introduced \emph{strongly
irreducible} Heegaard splittings, showing that weak reducibility forces an incompressible surface---a property that can then be exploited, via Gabai's work~\cite{Gabai1983}, to obtain a taut foliation.  On the other side of the rigidity spectrum, Ozsv{\'a}th and Szab{\'o}~\cite{OzsvathSzabo2004} proved that an $L$-space admits no coorientable taut foliation, so the $\alpha$/$\beta$ curve configuration on a Heegaard diagram can obstruct tautness altogether.  These results indicate that the relationship between Heegaard diagram combinatorics and foliation geometry is both deep and subtle.

The \emph{foliation realization problem}---synthesising a foliation from Heegaard diagram data---has seen significant recent progress. Li~\cite{Li2022} constructed taut foliations directly from genus-two Heegaard diagrams using branched surfaces and left-orderability of the fundamental group; he noted the difficulty of generalising the required spatial deformations to higher genera.

In this paper we take a different, genus-independent approach, using Gabai's sutured manifold theory~\cite{Gabai1983} as the engine.  We introduce the \emph{meridional sutured handlebody}, a sutured manifold whose sutures are prescribed by a disk system in a handlebody.  We prove (Proposition~\ref{prop:msh-foliation}) that every such handlebody decomposes, via basic sutured decomposition operations, into one Reeb component together with product disks.  The proof is a purely combinatorial induction on a complexity measure.

This leads to a three-step construction:
\begin{enumerate}
  \item Translate a Heegaard diagram into two meridional sutured handlebodies;
  \item Glue them along compatible disk regions by reversing disk decompositions;
  \item Fill the remaining cavities with a meridional sutured handlebody and some Reeb components.
\end{enumerate}
The construction applies to Heegaard diagrams of arbitrary genus, and each step involves only finitely many choices.

Our main result is:

\begin{maintheorem}\label{thm:main}
  Any closed, connected, orientable $3$-manifold admits a coorientable foliation constructed directly from its Heegaard diagram via the above three steps. 
\end{maintheorem}

\begin{note}
1. To the best of our knowledge, this is the first such construction for Heegaard diagrams of arbitrary genus.

2. The construction does not yield a taut foliation: the decomposition terminates at a Reeb component, whose sutured Floer homology vanishes.
\end{note}

The paper is organised as follows.  Section~2 recalls the necessary background on Heegaard diagrams and sutured manifolds. Section~3 introduces meridional sutured handlebodies and proves Proposition~\ref{prop:msh-foliation}. Section~4 proves the Main Theorem by carrying out the three-step construction.

\subsection*{Acknowledgements}
We have learned a lot in the virtual online weekly seminar for foliation theory organized by Yao Fan, Biao Ma and Bin Yu. We thank Tao Li, Yi Liu,  Ruifeng Qiu and Jiajun Wang for many helpful discussions; Ian Agol and David Gabai for many helpful suggestions during the CRM Conference---Knots, Groups, and Manifolds---held in 2025, 
and Steven Boyer for the invitation and hospitality.  This work was partially supported by the National Key R\&D program of China (2025YFA1017500), NSFC 12131009 \& 12471065, and in part by Science and Technology Commission of Shanghai Municipality (No.\ 22DZ2229014).

\section{Preliminaries}

\subsection{Heegaard Diagrams}

A \textit{handlebody} of genus~$g$ is obtained by attaching $g$ $1$-handles to a $3$-ball.  Two handlebodies of the same genus are homeomorphic.

\begin{definition}\label{def:heegaard-splitting}
Let $M$ be a closed, connected, oriented $3$-manifold. A \textit{Heegaard splitting} of genus $g$ is a decomposition $M = H_1\cup_{\Sigma} H_2$, where $H_1$ and $H_2$ are handlebodies of genus~$g$, and $\Sigma = \partial H_1 = \partial H_2$ is the \textit{Heegaard surface} (more precisely, $M$ is obtained from $H_1\sqcup H_2$ by a homeomorphism $\partial H_1\to\partial H_2$).
\end{definition}

A handlebody is determined up to homeomorphism by a complete collection of properly embedded, pairwise disjoint, non-separating essential disks whose boundaries cut the surface into a planar surface; such a collection is called a \textit{complete meridian system}.  A Heegaard splitting is therefore encoded by the boundary curves of these disks.

\begin{definition}\label{def:heegaard-diagram}
  A \textit{Heegaard diagram} is a triple $(\Sigma,\alpha,\beta)$, where $\Sigma$ is a closed oriented surface of genus~$g$, and
  $\alpha = \{\alpha_1,\dots,\alpha_g\}$ and $\beta  = \{\beta_1,\dots,\beta_g\}$ are complete meridian systems for $H_1$ and $H_2$ respectively.  Thus:
  \begin{enumerate}
    \item The curves in $\alpha$ are pairwise disjoint and non-parallel, and likewise for $\beta$.
    \item Cutting $\Sigma$ along $\alpha$ (resp.\ $\beta$) yields a $2$-sphere with $2g$ boundary components.
    \item Each $\alpha_i$ (resp.\ $\beta_i$) bounds a properly embedded essential disk in $H_1$ (resp.\ $H_2$).
  \end{enumerate}
\end{definition}

In Section~4 we will enlarge a given diagram to a \textit{complete-and-separating} diagram $D_0$ containing both
non-separating and separating curves, for which $\Sigma\setminus D_0$ is a union of open disks; this is a convenient working diagram obtained by adding curves to a Heegaard diagram.

\subsection{Sutured Manifolds}

We recall Gabai's sutured manifold theory~\cite{Gabai1983}.  All manifolds in this paper are compact and
oriented.

\begin{definition}[Gabai~\cite{Gabai1983}]\label{def:sutured-manifold}
  A \textit{sutured manifold} $(M,\gamma)$ consists of a $3$-manifold~$M$ and a collection $\gamma\subset\partial M$ of disjoint annuli
  $A(\gamma)$ and tori $T(\gamma)$.  Each annulus contains a \textit{suture}---an oriented, homologically non-trivial simple
  closed curve; the set of sutures is denoted $s(\gamma)$. The complement $R(\gamma)=\partial M\setminus\operatorname{Int}(\gamma)$  decomposes as $R_+(\gamma)\sqcup R_-(\gamma)$, where $R_+(\gamma)$ is oriented so that its normal vector points out of~$M$ and $R_-(\gamma)$ points into~$M$; these orientations are consistent with those of $s(\gamma)$ along $\partial R(\gamma)$. For convenience, \(R_+(\gamma)\) and \(R_-(\gamma)\) are sometimes abbreviated to \(R_+\) and \(R_-\), respectively.
\end{definition}

A sutured manifold $(M,\gamma)$ \textit{admits a co-oriented codimension-one foliation} $\mathcal{F}$ if:
\begin{enumerate}
  \item $\mathcal{F}$ is transverse to $\gamma$;
  \item $\mathcal{F}$ is tangent to $R(\gamma)$, with leaves near $R_+(\gamma)$ pointing outward and near $R_-(\gamma)$ pointing inward, matching the orientations in Definition~\ref{def:sutured-manifold};
  \item $\mathcal{F}|_{\gamma}$ contains no Reeb components.
\end{enumerate}

The following two sutured manifolds are the fundamental building blocks (see Figure~\ref{fig_1_1_product_disk_and_reeb}; our convention: red/purple~=~sutures, blue~=~$R_+(\gamma)$, green~=~$R_-(\gamma)$).

\begin{example}\label{ex:product-disk}
  \textbf{Product disk.} $M = D^2\times[0,1]$, with $\gamma = \partial D^2\times[0,1]$ (a single annular suture). The leaves are $D^2\times\{\text{pt}\}$, $R_+ = D^2\times\{1\}$, $R_- = D^2\times\{0\}$.
\end{example}

\begin{example}\label{ex:reeb}
  \textbf{Reeb component.} A solid torus with a foliation consisting of planar leaves spiralling toward the boundary torus.  Here $\gamma=\varnothing$; the boundary is either $R_+$ or $R_-$.
\end{example}

\begin{figure}[H]
  \centering
  \includegraphics[width=0.6\linewidth]{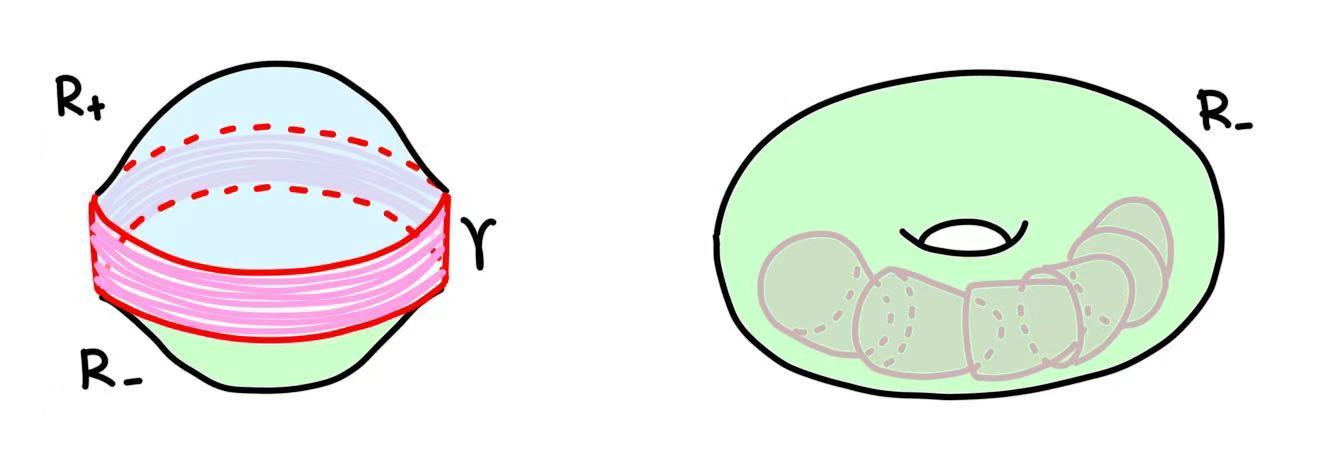}
  \caption{Product disk and Reeb component}
  \label{fig_1_1_product_disk_and_reeb}
\end{figure}

The definition of sutured manifold decomposition is introduced below.

\begin{definition}[Gabai~\cite{Gabai1983}]
\label{def:sutured manifold decomposition}
    Let $(M,\gamma)$ be a sutured manifold and $S$ a properly embedded surface in $M$ such that for every component $\lambda$ of $S\cap\gamma$ one of (1)--(3) holds:
\begin{enumerate}
    \item $\lambda$ is a properly embedded nonseparating arc in $\gamma$.
    \item $\lambda$ is a simple closed curve in an annular component $A(\gamma)$ of $\gamma$ in the same homology class as $A(\gamma)\cap s(\gamma)$.
    \item $\lambda$ is a homotopically nontrivial curve in a toral component $T(\gamma)$ of $\gamma$, and if $\delta$ is another component of $T(\gamma)\cap S$, then $\lambda$ and $\delta$ represent the same homology class in $H_1(T(\gamma))$.
\end{enumerate}
The surface $S$ defines a sutured manifold decomposition
\[
(M,\gamma)\stackrel{S}{\leadsto}(M',\gamma'),
\]
where $M' = M \setminus \operatorname{int}(N(S))$ and
\[
\gamma' = \bigl(\gamma\cap M'\bigr)
           \cup N(S'_+\cap R_-(\gamma))
           \cup N(S'_-\cap R_+(\gamma)),
\]
\[
R'_+(\gamma') = \bigl((R_+(\gamma)\cap M')\cup S'_+\bigr)\setminus \operatorname{int}(\gamma'),
\]
\[
R'_-(\gamma') = \bigl((R_-(\gamma)\cap M')\cup S'_-\bigr)\setminus \operatorname{int}(\gamma'),
\]
where $S'_+$ (resp.\ $S'_-$) is that component of $\partial N(S)\cap M'$ whose normal vector points out of (resp.\ into) $M'$.
\end{definition}

We now introduce three basic decomposition operations that simplify sutured manifolds. Throughout the subsequent discussion, we only consider the case (1) for the intersection \(S\cap\gamma\) in the decomposition definition~\ref{def:sutured manifold decomposition}. A foliation on the decomposed manifold can be lifted back to the original one by Lemma~\ref{lemma:decomposition}. In the figures below, $\longrightarrow$ (black) denotes a decomposition, while $\longleftarrow$ (red) indicates that a foliation can be reconstructed from the decomposed manifold.

\begin{operation}\label{op:trivial}
  \textbf{Trivial decomposition.} Suppose there exists a properly embedded disk \(D\subset M\) such that $(M,\gamma)\stackrel{D}{\leadsto}(M',\gamma')$, where \(|\partial D \cap s(\gamma)|=2\), and \(D\) is non-separating, i.e., the manifold \(M\setminus N(D)\) is connected. See Figure~\ref{fig_1_2_operation_1}.
\end{operation}

\begin{figure}[H]
  \centering
  \includegraphics[width=0.8\linewidth]{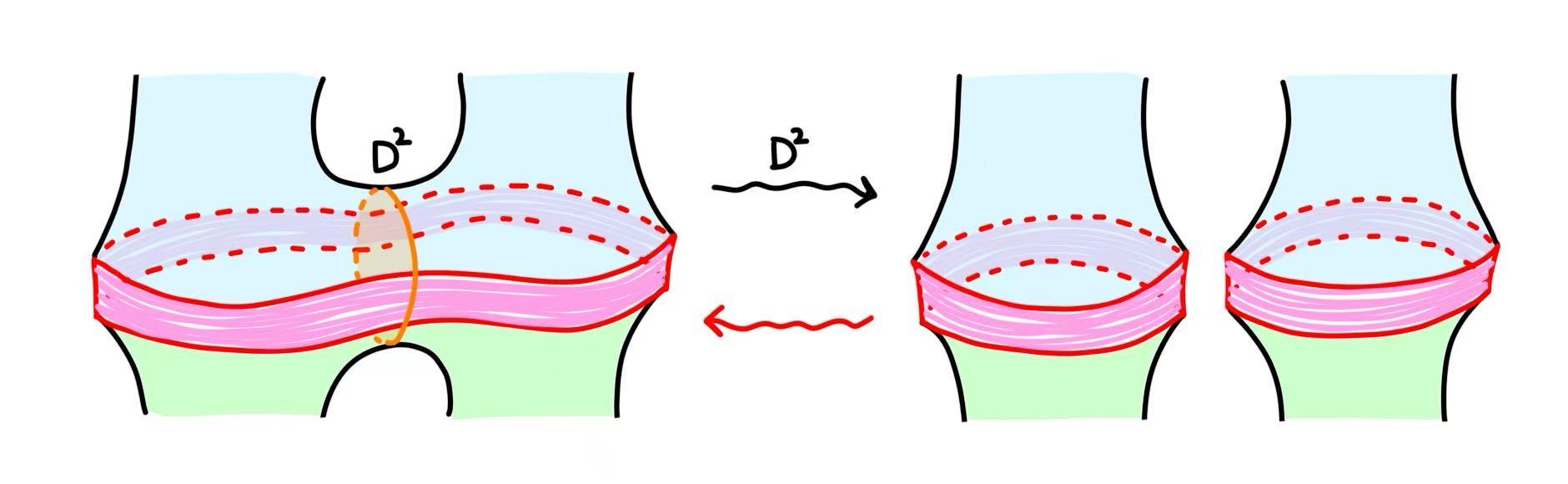}
  \caption{Trivial decomposition}
  \label{fig_1_2_operation_1}
\end{figure}

\begin{operation}\label{op:disk}
 \textbf{Disk decomposition.} Suppose there exists a properly embedded disk $D\subset M$ such that $(M,\gamma)\stackrel{D}{\leadsto}(M',\gamma')$, where $|\partial D \cap s(\gamma)|=2n,n\ge 1$. See Figure~\ref{fig_1_3_operation_2}.
\end{operation}

\begin{figure}[H]
  \centering
  \includegraphics[width=0.8\linewidth]{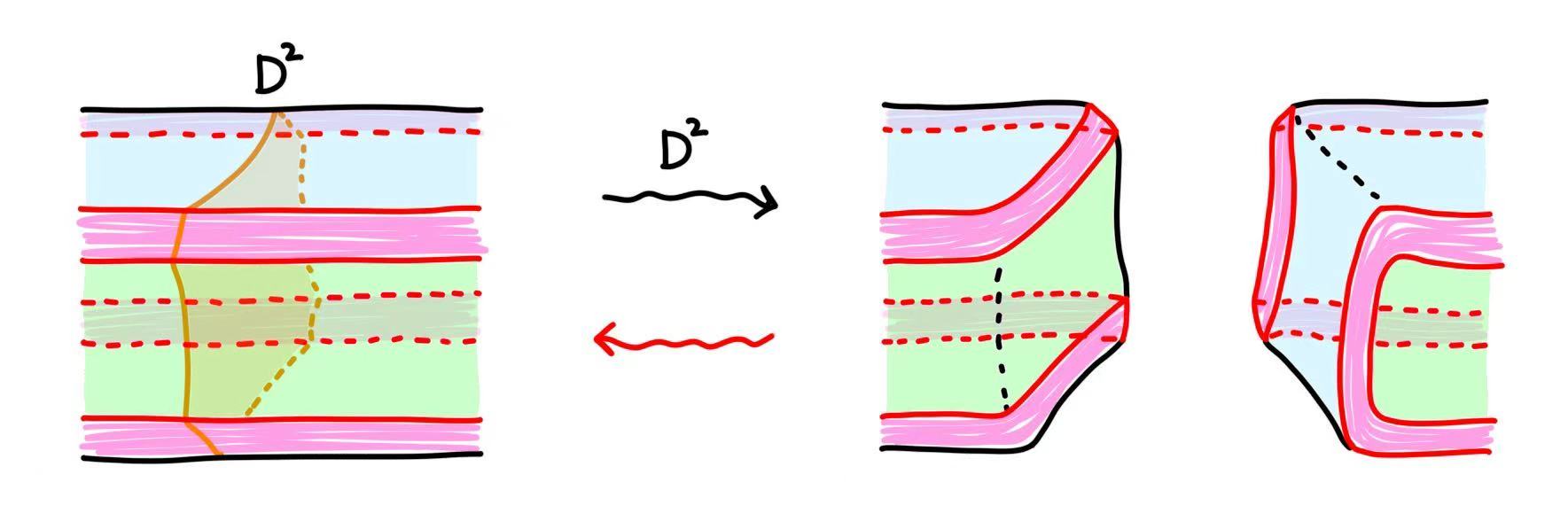}
  \caption{Disk decomposition}
  \label{fig_1_3_operation_2}
\end{figure}

A trivial decomposition is a special case of a disk decomposition; we distinguish it because it appears frequently.

\begin{example}\label{ex:monkey-saddle}
\textbf{Solid torus with a monkey saddle.} This is a sutured solid torus with \(2n\) longitudinal sutures for \(n\ge 1\), such that \(s(\gamma)\) meets each meridian in \(2n\) points. It can be reduced to a product disk by one disk decomposition (cutting along a meridian disk disjoint from \(s(\gamma)\)); hence its foliation is inherited from the product disk. Figure~\ref{fig_1_4_example_3} shows the case \(n=2\).
\end{example}

\begin{figure}[H]
  \centering
  \includegraphics[width=0.8\linewidth]{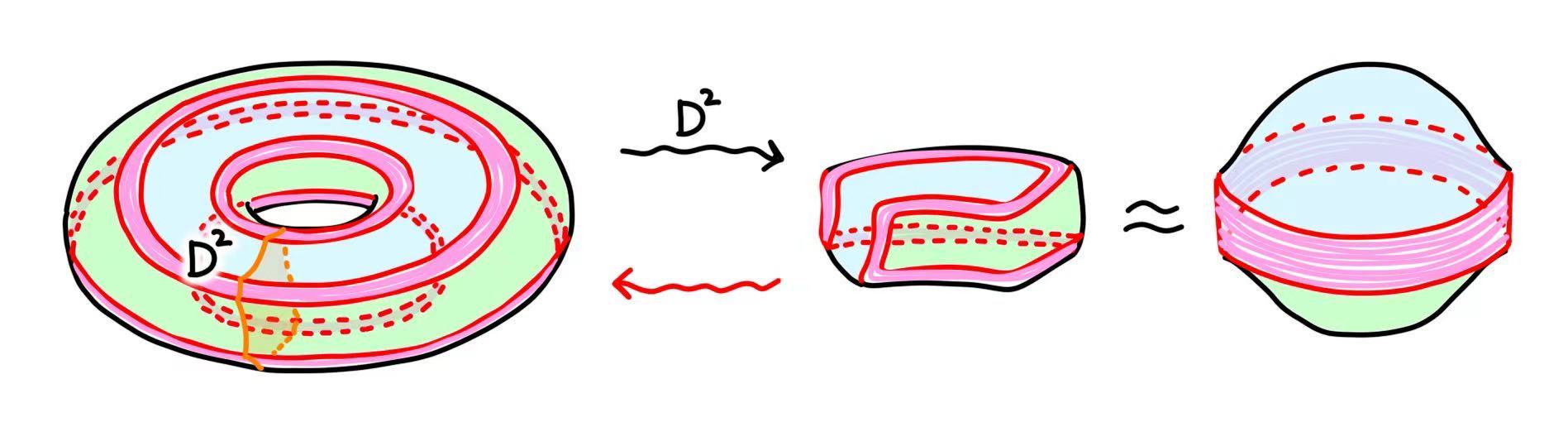}
  \caption{The solid torus with a monkey saddle}
  \label{fig_1_4_example_3}
\end{figure}

\begin{operation}\label{op:annular}
  \textbf{Annular decomposition.} Suppose there exists a properly embedded \emph{boundary-parallel} annulus \(A\subset M\) such that $(M,\gamma)\stackrel{A}{\leadsto}(M',\gamma')$, where both components of \(\partial A\) lie in \(R_{+}\) or both lie in \(R_{-}\). See Figure~\ref{fig_1_5_operation_3}.
\end{operation}

\begin{figure}[H]
  \centering
  \includegraphics[width=0.7\linewidth]{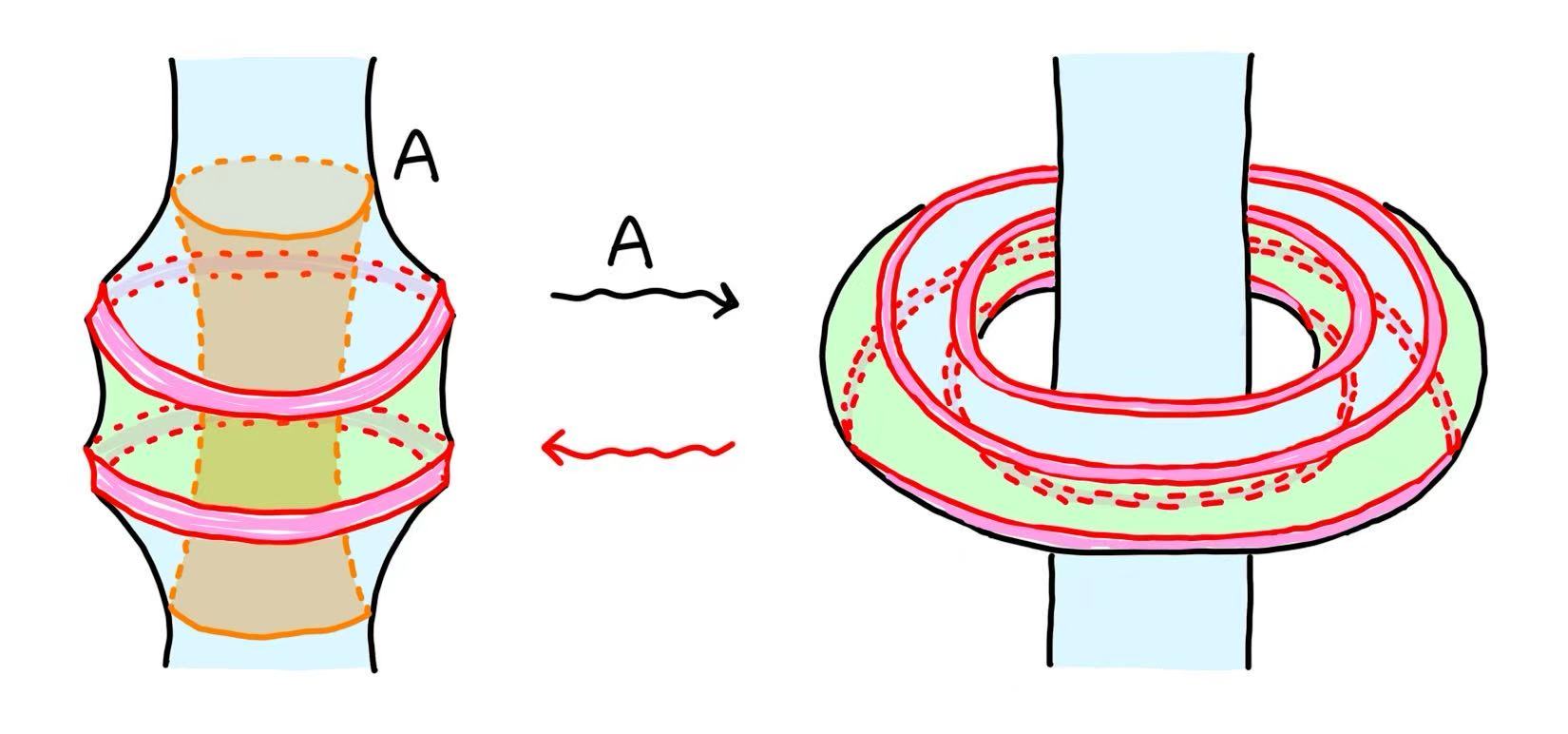}
  \caption{Annular decomposition}
  \label{fig_1_5_operation_3}
\end{figure}

An annular decomposition separates a monkey saddle solid torus (Example~\ref{ex:monkey-saddle}) from the remainder.  Conversely, if a monkey saddle solid torus carries a foliation and is attached to a Reeb component along a boundary annulus, its foliation can be adjusted to extend across the gluing.

\begin{example}\label{ex:meridional-solid-torus}
  \textbf{Meridional sutured solid torus.}  A solid torus with $2n$ meridional sutures, $n\ge 1$, each bounding a meridian disk. See Figure~\ref{fig_1_6_example_4} for the case $n=1$.  The same decomposition applies for any~$n$.
\end{example}

\begin{figure}[H]
  \centering
  \includegraphics[width=0.8\linewidth]{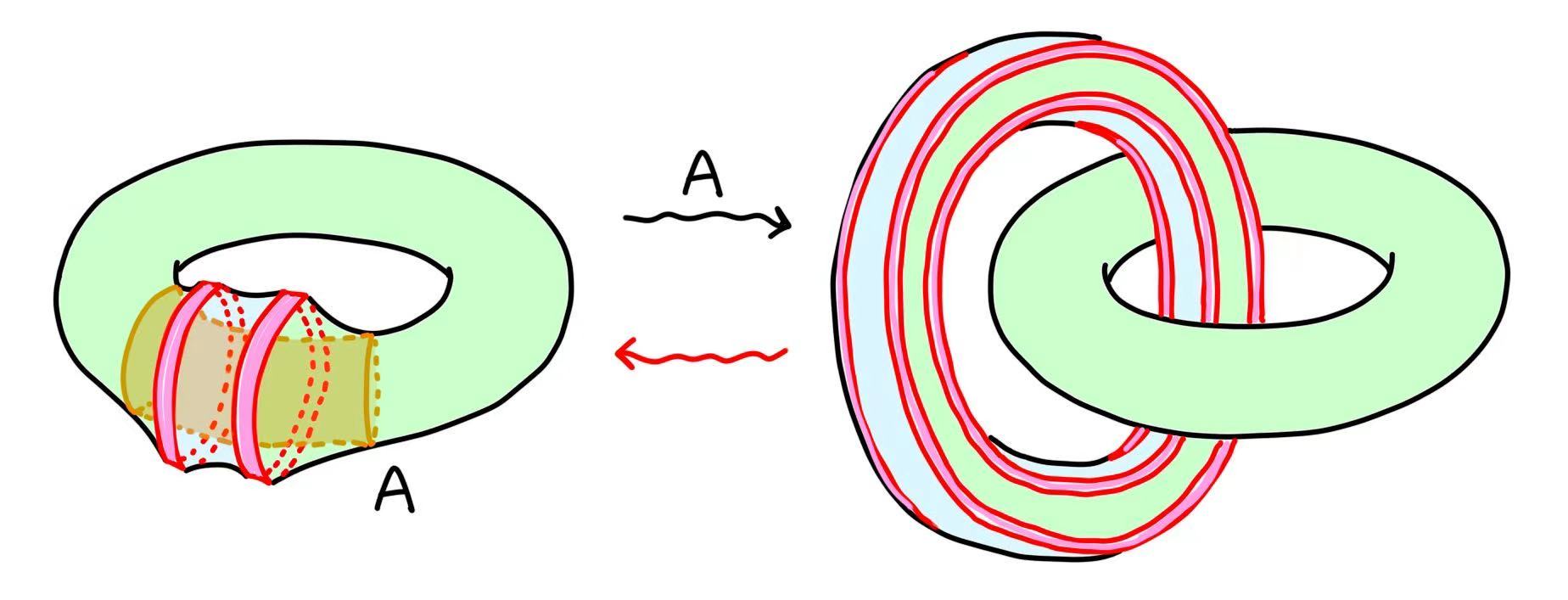}
  \caption{A meridional sutured solid torus}
  \label{fig_1_6_example_4}
\end{figure}

\begin{lemma}[Gabai~\cite{Gabai1983}]\label{lemma:decomposition}
  For a sutured manifold decomposition $(M,\gamma)\overset{S}{\leadsto}(M',\gamma')$, if $(M',\gamma')$ admits a foliation, then $(M,\gamma)$ also admits a foliation.
\end{lemma}

Lemma~\ref{lemma:decomposition} is the key property that makes sutured decompositions useful for constructing foliations: it allows us to reduce the problem to building a foliation on the simplest pieces and then reconstructing the foliation on the original manifold step by step.

\section{Construction of Meridional Sutured Handlebodies}
\label{sec:msh}

We first define a structure graph of a handlebody.

\begin{definition}\label{def:structure-graph}
  Let $H$ be a handlebody of genus $g$, and let $D$ be a collection of properly embedded, pairwise disjoint disks in $H$.  The pair $(H,D)$ determines a graph $\Gamma(H,D)$, the \emph{structure graph} of $H$ relative to $D$.  The components of
  $H\setminus\bigcup_{D_i\in D} D_i$ are sub-handlebodies, corresponding to the vertices of $\Gamma(H,D)$; the disks in $D$ correspond to edges; and the genus of each sub-handlebody is assigned as a weight to its corresponding vertex.
\end{definition}

A vertex $v\in\Gamma(H,D)$ is \textit{isolated} if $\operatorname{edge}(v)=1$ and $\operatorname{weight}(v)=0$.

\begin{lemma}\label{lemma:isolated}
  For the pair $(H,D)$, every disk $D_i\in D$ is essential if and only if the graph $\Gamma(H,D)$ contains no isolated vertices.
\end{lemma}

We now define the \textit{meridional sutured handlebody} and show that it can be decomposed into a union of one Reeb component and several product disks.

\begin{definition}\label{def:msh}
  A \textit{meridional sutured handlebody} is a sutured manifold $(H,\gamma)$, where $H$ is a handlebody of genus $g\ge 1$ and $\gamma$ (possibly empty) consists entirely of annular components. Furthermore, for each suture component $s_i\in s(\gamma)$ there exists a properly embedded disk $D_i\subset H$ with $\partial D_i = s_i$.  Let $D(\gamma)$ denote the collection of these pairwise disjoint disks. The manifold $(H,\gamma)$ is required to satisfy:
  \begin{enumerate}
    \item The structure graph $\Gamma(H,D(\gamma))$ contains no isolated vertices (see Figure~\ref{fig_2_1_sutured_handlebody});
    \item $\chi(R_+(\gamma)) = \chi(R_-(\gamma))$, with the convention that $\chi(\emptyset) = 0$.
  \end{enumerate}
  The complexity of $(H,\gamma)$ is $c(H,\gamma) = \bigl(g(H),\,\#s(\gamma)\bigr)$, ordered lexicographically.
\end{definition}

\begin{figure}[H]
  \centering
  \includegraphics[width=1\linewidth]{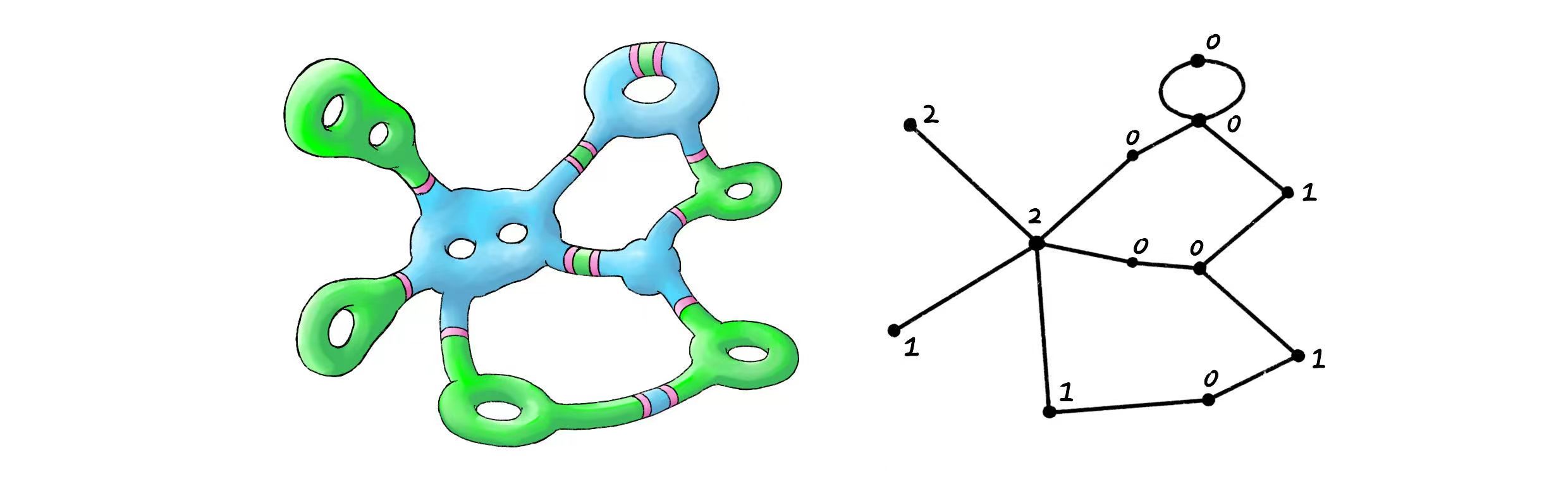}
  \caption{A meridional sutured handlebody and its structure graph}
  \label{fig_2_1_sutured_handlebody}
\end{figure}

\begin{remark}
  $c(\text{Reeb component}) = (1,0)$.
\end{remark}

\begin{lemma}\label{lemma:reeb}
  For any meridional sutured handlebody $(H,\gamma)$, the following are equivalent:
  \begin{enumerate}
    \item $\#s(\gamma) = 0$;
    \item $(H,\gamma)$ is a Reeb component.
  \end{enumerate}
\end{lemma}

\begin{proposition}\label{prop:msh-foliation}
Every meridional sutured handlebody $(H,\gamma)$ admits a foliation. Moreover, this foliation can be constructed from one Reeb component and several product disks via decomposition operations \ref{op:trivial},~\ref{op:disk} and~\ref{op:annular}.
\end{proposition}

\begin{proof}
  We decompose $(H,\gamma)$ into a Reeb component and several product disks.  For convenience we denote the component under consideration during the decomposition process by $(H,\gamma)$.

  If $\#s(\gamma)=0$, then by Lemma~\ref{lemma:reeb} the manifold is a Reeb component, which carries a Reeb foliation, and the proof terminates.  Otherwise we perform one of the following operations:

  \medskip\noindent\textbf{Case 1.}
  Suppose there exists $v\in\Gamma(H,D(\gamma))$ with $\operatorname{edge}(v)=2$ and $\operatorname{weight}(v)=0$.

  We apply an annular decomposition (Figure~\ref{fig_1_5_operation_3}), which separates a monkey saddle
  solid torus (further decomposable into a disk); let $(H',\gamma')$ denote the remaining part.  The complexity decreases: $c(H',\gamma') < c(H,\gamma)$, because $g(H') = g(H)$ and $\#s(\gamma') = \#s(\gamma)-2$.

  There are two subcases, depicted in the left of Figure~\ref{fig_2_1_2_local_structure}.  In both, $\Gamma(H',D(\gamma'))$ contains no isolated vertices: removing the two edges incident to the degree-2 vertex (together with the vertex itself, whose weight is zero) cannot create new isolated vertices, because the two neighbouring vertices lose one incident edge each but retain at least one other edge by construction. Moreover, this operation reduces one annular component from either $R_+$ or $R_-$ and attaches it to the other, so $\chi(R_+(\gamma')) = \chi(R_-(\gamma'))$ holds. Hence $(H',\gamma')$ satisfies Definition~\ref{def:msh}.

  \begin{figure}[H]
    \centering
    \includegraphics[width=0.9\linewidth]{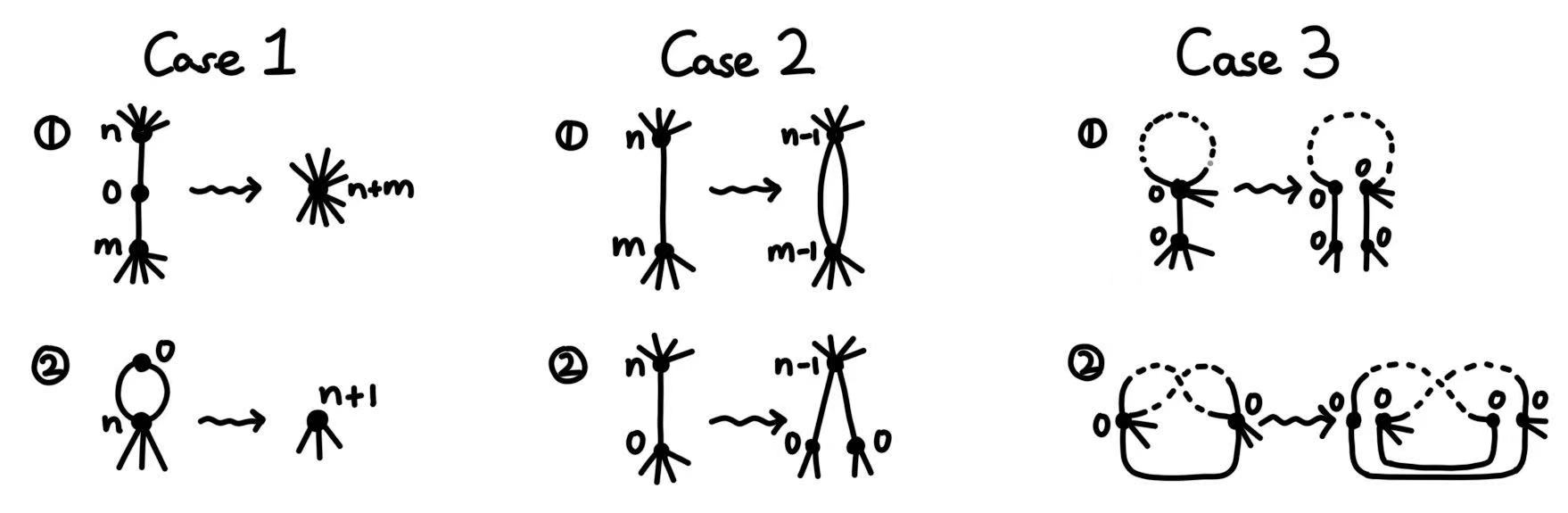}
    \caption{A pictorial illustration of the decomposition process}
    \label{fig_2_1_2_local_structure}
  \end{figure}

  \medskip\noindent\textbf{Case 2.}
  Suppose there exists $v\in\Gamma(H,D(\gamma))$ with $\operatorname{weight}(v)>0$, and Case~1 does not apply to any vertex.

  There are two feasible trivial decompositions, illustrated in Figure~\ref{fig_2_2_step_2}.  After performing the decomposition, let $(H',\gamma')$ denote the remaining part.  The complexity decreases: $c(H',\gamma') < c(H,\gamma)$, since $g(H') = g(H)-1$.

  The configurations are shown in the middle of Figure~\ref{fig_2_1_2_local_structure}. In both cases, $\Gamma(H',D(\gamma'))$ contains no isolated vertices, because the trivial decomposition modifies a pair of adjacent vertices (one of positive weight and one of its neighbors), so that the vertex undergoing the change has degree at least $2$. Again $\chi(R_+(\gamma')) = \chi(R_-(\gamma'))$ by symmetry of the decomposition.  Hence $(H',\gamma')$ satisfies Definition~\ref{def:msh}.

  \begin{figure}[H]
    \centering
    \includegraphics[width=0.7\linewidth]{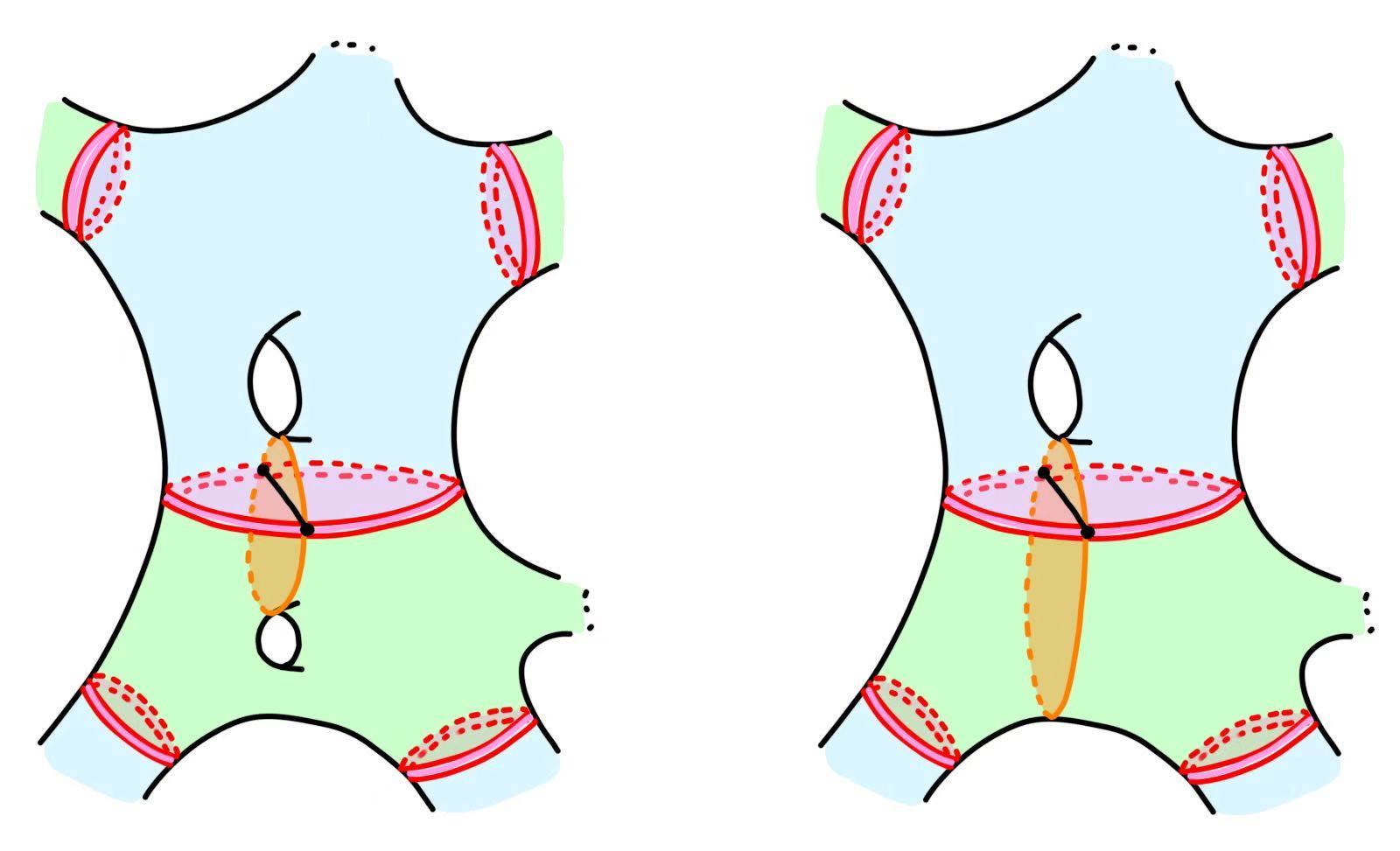}
    \caption{Trivial decompositions in Case~2}
    \label{fig_2_2_step_2}
  \end{figure}

  \medskip\noindent\textbf{Case 3.}
  Suppose that for every vertex $v\in\Gamma(H,D(\gamma))$ we have $\operatorname{edge}(v)>2$ and $\operatorname{weight}(v)=0$.

  Then $\Gamma(H,D(\gamma))$ must contain a cycle.  Let $C$ be such a cycle and choose a vertex $v$ on $C$.  Since $\operatorname{edge}(v)\ge 3$, there exists an edge incident to $v$ that does not belong to~$C$. Let $u$ be the vertex at the other end of this edge.  If $u$ is not on the cycle~$C$, we apply the trivial decomposition shown in the left panel of Figure~\ref{fig_2_3_step_3}; if $u$ lies on~$C$, we apply the one in the right panel.

  Let $(H',\gamma')$ denote the remaining part after the decomposition. The configurations are shown in the right of Figure~\ref{fig_2_1_2_local_structure}. The complexity decreases: $c(H',\gamma') < c(H,\gamma)$, since $g(H') = g(H)-1$.  In both subcases, $\Gamma(H',D(\gamma'))$ contains no isolated vertices: cutting along an edge not belonging to the cycle cannot render the remaining vertices isolated, because every vertex still belongs to a cycle or has degree $\ge 2$.  The Euler-characteristic equality $\chi(R_+(\gamma')) = \chi(R_-(\gamma'))$ is preserved by symmetry of the trivial decomposition.  Hence $(H',\gamma')$ satisfies Definition~\ref{def:msh}.

  \begin{figure}[H]
    \centering
    \includegraphics[width=0.8\linewidth]{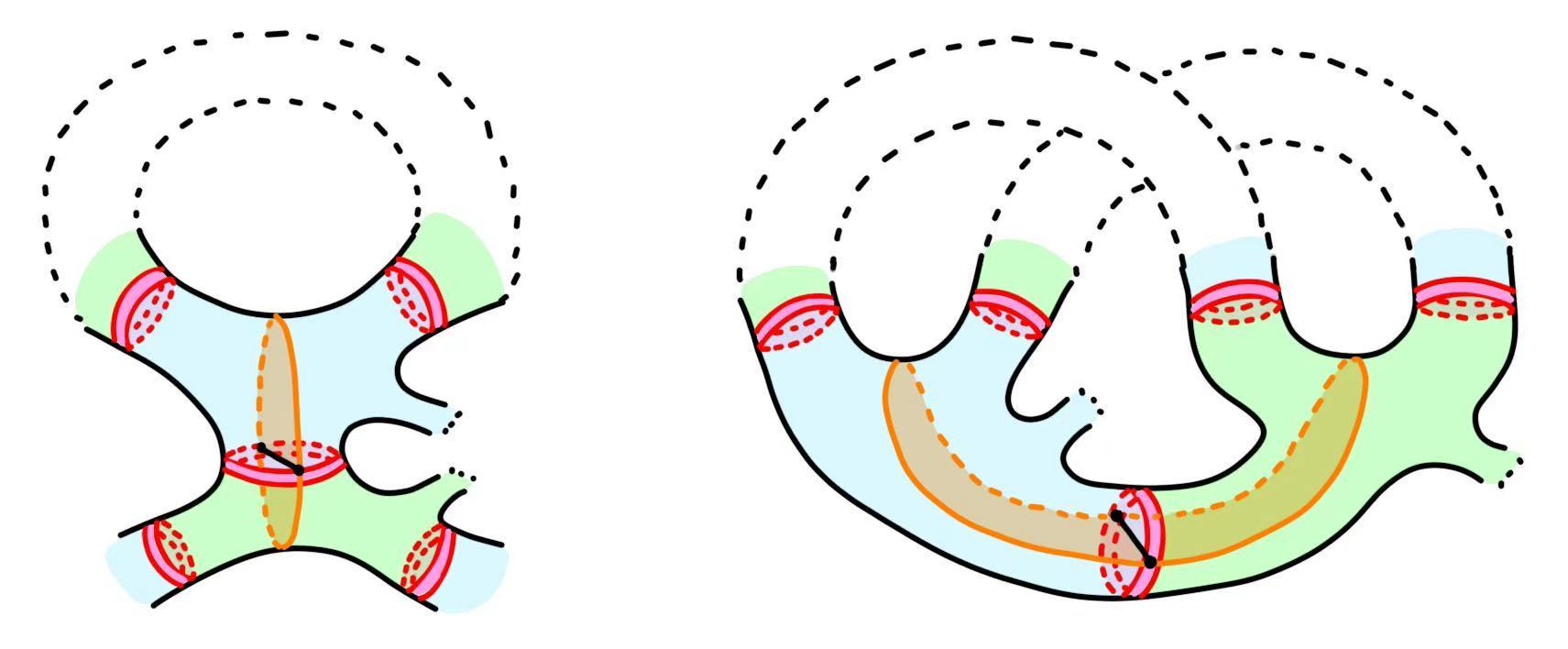}
    \caption{Trivial decompositions in Case~3}
    \label{fig_2_3_step_3}
  \end{figure}

  \medskip
  If $\#s(\gamma)\neq 0$ for a meridional sutured handlebody, then one of the above operations can always be performed.  By the well-ordering of the complexity, after finitely many steps we reach $\#s(\gamma)=0$, which is a Reeb component by Lemma~\ref{lemma:reeb}, together with several product disks generated along the way. Thus we construct a sequence of sutured manifold decompositions
  \[
    (H,\gamma) = (M_0,\gamma_0)
    \overset{S_1}{\leadsto} (M_1,\gamma_1)
    \overset{S_2}{\leadsto} \cdots
    \overset{S_n}{\leadsto} (M_n,\gamma_n)
  \]
  such that $(M_n,\gamma_n)$ is a Reeb component together with product disks.  By Lemma~\ref{lemma:decomposition}, $(H,\gamma)$ admits a foliation.
\end{proof}

\begin{remark}
  The existence of such a foliation can also be proved by other methods; here we aim to demonstrate the constructive nature of the foliation we construct.
\end{remark}

\section{Proof of the Main Theorem}
\label{sec:proof}

Suppose a closed, connected, orientable 3-manifold $M$ admits a Heegaard splitting $M=V\cup_{\Sigma} W$, where $\Sigma$ is equipped with a Heegaard diagram $D_0=(\Sigma, \alpha, \beta)$. Each genus-$g\ (g\ge 1)$ handlebody induces $2g-1$ curves on the diagram, including $g$ non-separating curves and $g-1$ separating curves that are pairwise non-parallel, as illustrated for the genus-two case in the upper part of Figure \ref{fig_3_1_Heegaard_diagram}. We assume that each component of $\Sigma\setminus D_0$ is an open disk, which can be achieved by a small perturbation of the original diagram. Note that a more general finger move method in \cite{SucharitWang2010} ensures that the complement of the perturbed diagram consists of open disks. This assumption is adopted merely for simplicity, while the subsequent construction remains fully valid even if $\Sigma\setminus D_0$ contains non-disk components. We construct a foliation of $M$ via the following three steps.

\subsection{Step~1: Construction of the handlebodies}

We first construct a new diagram $D'_0$ and then define two disjoint meridional sutured handlebodies $V'\sqcup W'\subset M$ (as submanifolds of $M$), as shown in the lower part of Figure \ref{fig_3_1_Heegaard_diagram}.

The diagram $D_0'=(\Sigma, \alpha', \beta')$ is constructed as follows. Let $\alpha=\{\alpha_1,\cdots,\alpha_g\}\cup\{\alpha_{g+1},\cdots,\alpha_{2g-1}\}$, where $\alpha_1,\cdots,\alpha_g$ denote nonseparating curves and $\alpha_{g+1},\cdots,\alpha_{2g-1}$ denote separating curves. We define $$\alpha'=\{\partial N(\alpha_1),\partial N(\alpha_2),\cdots,\partial N(\alpha_g),\alpha_{g+1},\cdots,\alpha_{2g-1}\},$$ which means that each nonseparating curve is replaced by a pair of parallel copies of the original curve. The same operation is applied to the set $\beta$ to obtain $\beta'$. Still each component of $\Sigma\setminus D'_0$ is an open disk.

And the meridional sutured handlebodies $V'\sqcup W'\subset M$ are defined as follows. Let $V'=\overline{V\setminus N(\Sigma)}$ and $W'=\overline{W\setminus N(\Sigma)}$. Then  $M$ admits the decomposition $M=(V'\sqcup W')\cup_{\Sigma\times \partial I}\Sigma\times I$ with $\Sigma\times\{0\}=\partial V'$ and $\Sigma\times\{1\}=\partial W'$.  

\begin{figure}[H]
    \centering
    \includegraphics[width=0.8\linewidth]{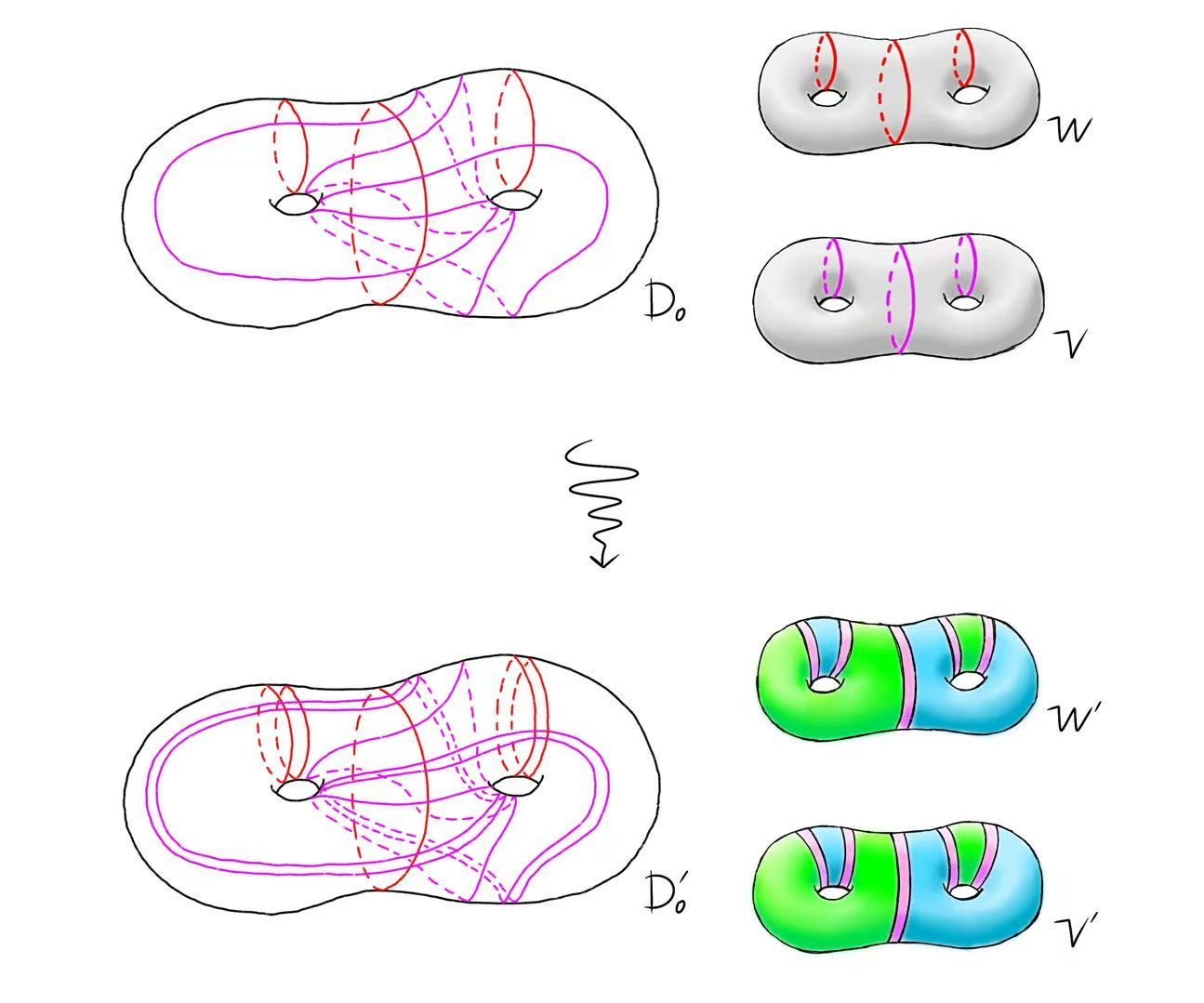}
    \caption{From handlebodies to meridional sutured handlebodies}
    \label{fig_3_1_Heegaard_diagram}
\end{figure}

Next, we equip these handlebodies with sutures. Since $\Sigma\times\{0\}=\partial V'$, the family of curves $\alpha'\times\{0\}$ lies on $\partial V'$. We define $(V',\gamma_V')$ by $\gamma_V'=N(\alpha'\times\{0\})$, and each component of $R(\gamma_V')=\partial V'\setminus\text{Int}{\gamma}_V'$ is orientable, where the vector field on $R_+(\gamma_V')$ points outward and that on $R_-(\gamma_V')$ points inward. These orientations are consistent with the intrinsic orientation of $s(\gamma_V')$. Similarly, since $\Sigma\times\{1\}=\partial W'$, the family of curves $\beta'\times\{1\}$ lies on $\partial W'$. We define $(W',\gamma_W')$ by $\gamma_W'=N(\beta'\times\{1\})$, and each component of $R(\gamma_W')=\partial W'\setminus\text{Int}{\gamma}_W'$ is orientable, with outward-pointing vector fields on $R_+(\gamma_W')$ and inward-pointing vector fields on $R_-(\gamma_W')$, whose orientations are compatible with those of $s(\gamma_W')$. For convenience in what follows, we denote the resulting sutured manifold by $(V'\sqcup W',\gamma')$. The above construction is illustrated in Figure~\ref{fig_3_1_Heegaard_diagram}.

So, both $V'$ and $W'$ are meridional sutured handlebodies.   In particular, it can be realized by applying the operation~\ref{op:trivial} to the $(4g+2)$-suture structure presented in Example \ref{ex:meridional-solid-torus}, as illustrated in Figure \ref{fig_3_2_handlebody}.

\begin{figure}[H]
    \centering
    \includegraphics[width=1\linewidth]{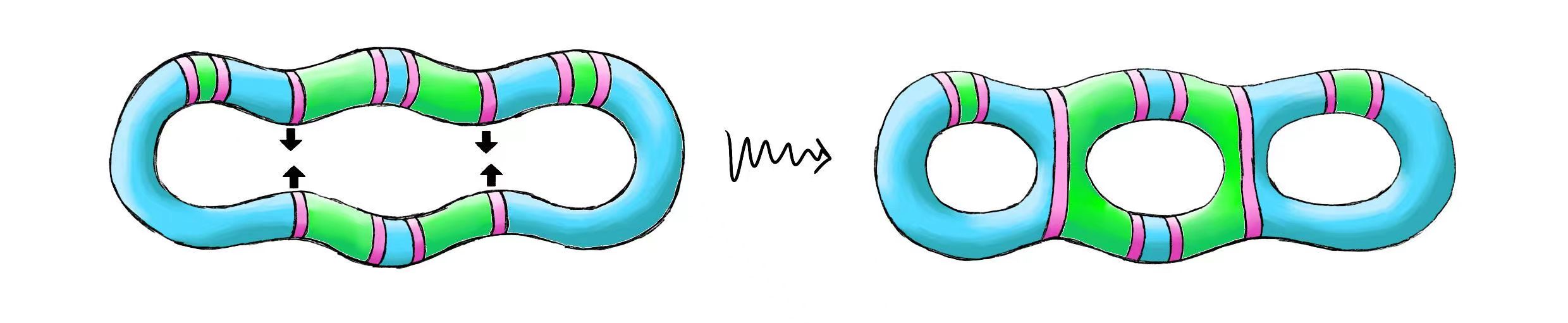}
    \caption{Construction of a genus-3 meridional sutured handlebody}
    \label{fig_3_2_handlebody}
\end{figure}

As a handlebody can be treated as a product I-bundle of a planar surface, we remark that alternative foliations for handlebodies are available, such as those induced by products of planar surfaces with $I$. Nevertheless, we adopt the aforementioned construction in this paper, since the resulting sutures are more closely correlated with the underlying Heegaard diagram. In particular, the genus-zero case can be regarded as two product disks glued together.

\subsection{Step~2: Gluing the handlebodies}

Remember that each region of $\Sigma\setminus D'_0$ is an open disk. Then for any open disk $d_0\subset \Sigma\setminus D'_0$, we classify it into two types according to the intersection property of its product boundary components with the $R(\gamma')$ regions: (I) The $d_0\times \{0\}$ and $d_0\times \{1\}$ have non-empty intersection with either $R_+(\gamma')$ or $R_-(\gamma')$; (II)
One of $d_0\times \{0\}$ and $d_0\times \{1\}$ has non-empty intersection with $R_+(\gamma')$, while the other has non-empty intersection with $R_-(\gamma')$. Note that both types of disks exist, since every four-valent vertex in $D_0'$ is surrounded by exactly two Type I disks and two Type II disks.

We apply the inverse operation of disk decomposition for a Type II disk, as demonstrated in Figure \ref{fig_3_3_step_2}. We select the disks \(d_V\) and \(d_W\) defined by \(d_V=((d_0\times\{0\})\cap R(\gamma'))\setminus N(\beta'\times \{0\})\) and \(d_W=((d_0\times\{1\})\cap R(\gamma'))\setminus N(\alpha'\times \{1\})\). Note that when performing this operation, we can slightly push \(d_V\) and \(d_W\) inward into \(d_0\), so that the $\gamma'$ on their boundaries are also pushed into the interior of \(d_0\). Consequently, the resulting manifold after this operation can be regarded as \(V'\cup W'\cup (U_x\times I)\), where \(x\in d_0\) and \(U_x\) denotes a sufficiently small neighborhood of $x$.

\begin{figure}[H]
    \centering
    \includegraphics[width=1\linewidth]{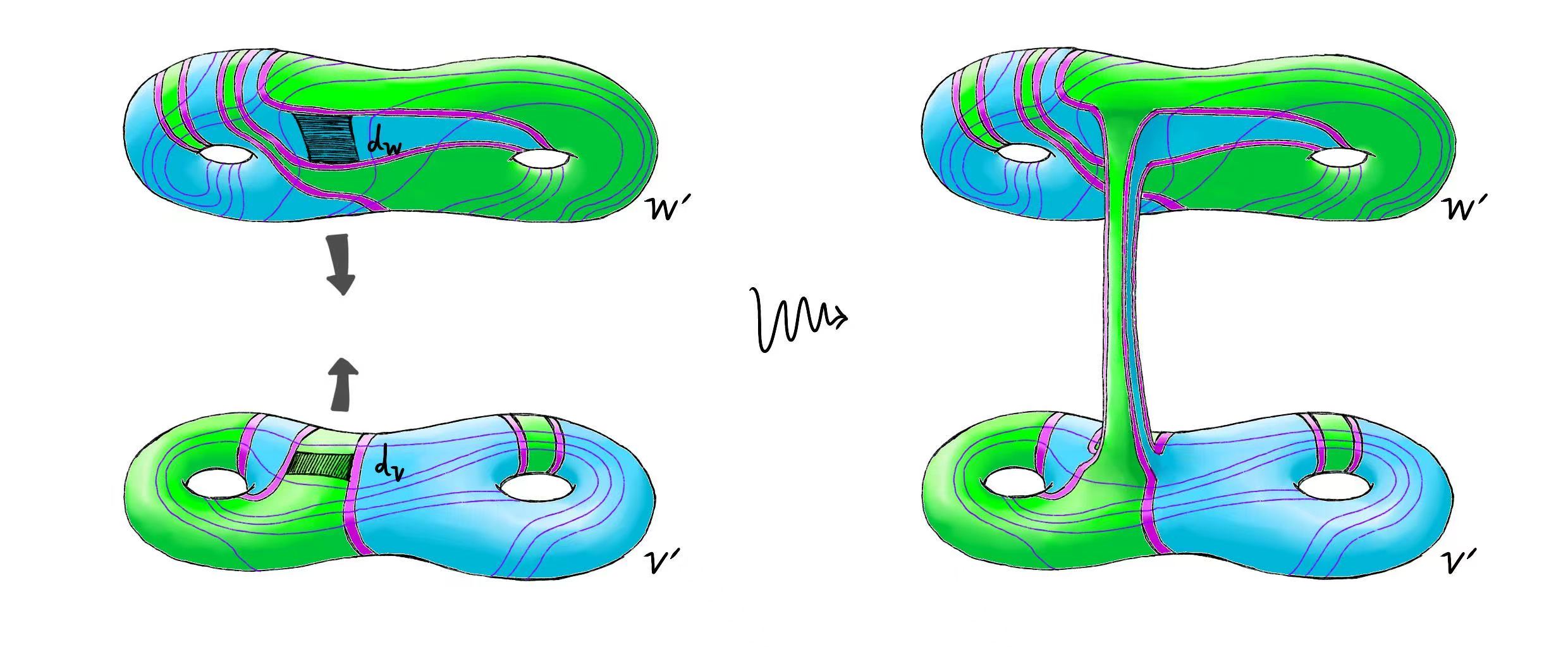}
    \caption{Local gluing of meridional sutured handlebodies}
    \label{fig_3_3_step_2}
\end{figure}

Note that the above local operation does not interfere with other Type II disks. Therefore, we can implement the aforementioned gluing operations sequentially or simultaneously for all Type II disk components. Denote the resulting submanifold of $M$ by $X$. Naturally, the resulting sutured manifold \((X,\gamma_X)\) admits a foliation, whose local geometry is illustrated in Figure~\ref{fig_3_3_2_step_2}. The depicted diagram is completed to a quadrilateral, and all other cases are handled analogously.

\begin{figure}[H]
    \centering
    \includegraphics[width=1\linewidth]{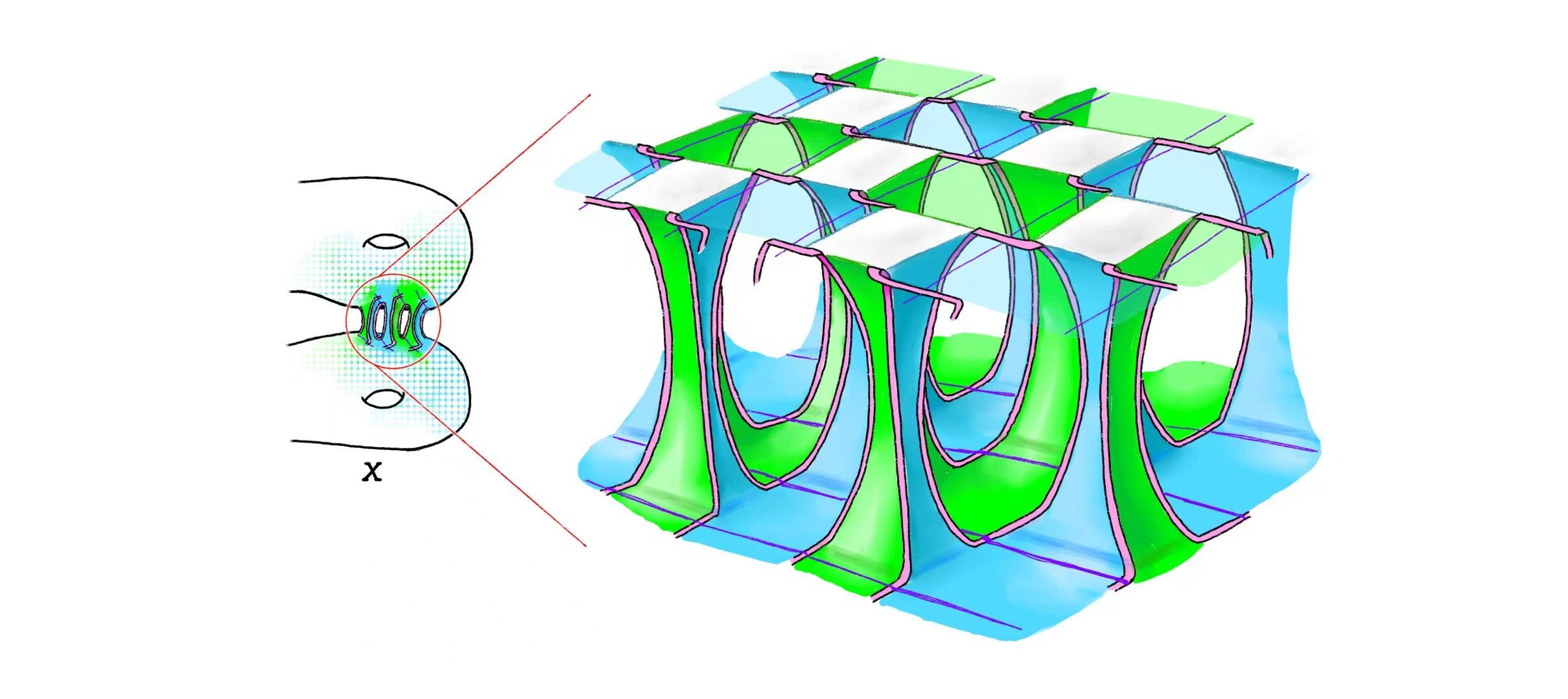}
    \caption{Local geometric configuration of $X$} 
    \label{fig_3_3_2_step_2} 
\end{figure}

\subsection{Step~3: Filling a meridional sutured handlebody and some Reeb components}

Let $H = M \setminus \operatorname{Int} N(X)$. Then the manifold $M$ admits the decomposition $M = (X \sqcup H) \cup_{\partial X \times \partial I} (\partial X \times I)$, where $\partial X \times \{0\} = \partial X$ and $\partial X \times \{1\} = \partial H$. Intuitively, $H$ can be regarded as a slight inward push-in of the complement of $X$.

We endow $H$ with a sutured manifold structure $(H, \gamma_H)$ as follows. If $\gamma_X \times \{0\} \subset \partial X \times \{0\}$ is the $\gamma_X$ of $X$, then we set $\gamma_H = \gamma_X \times \{1\}$ on $\partial H$. If $R_{\pm}(\gamma_X) \times \{0\} \subset \partial X \times \{0\}$ are the $R(\gamma_X)$ of $X$, then the $R(\gamma_H)$ of $H$ are defined by $R_{\mp}(\gamma_H) = R_{\pm}(\gamma_X) \times \{1\}$.

\begin{figure}[H]
    \centering
    \includegraphics[width=1\linewidth]{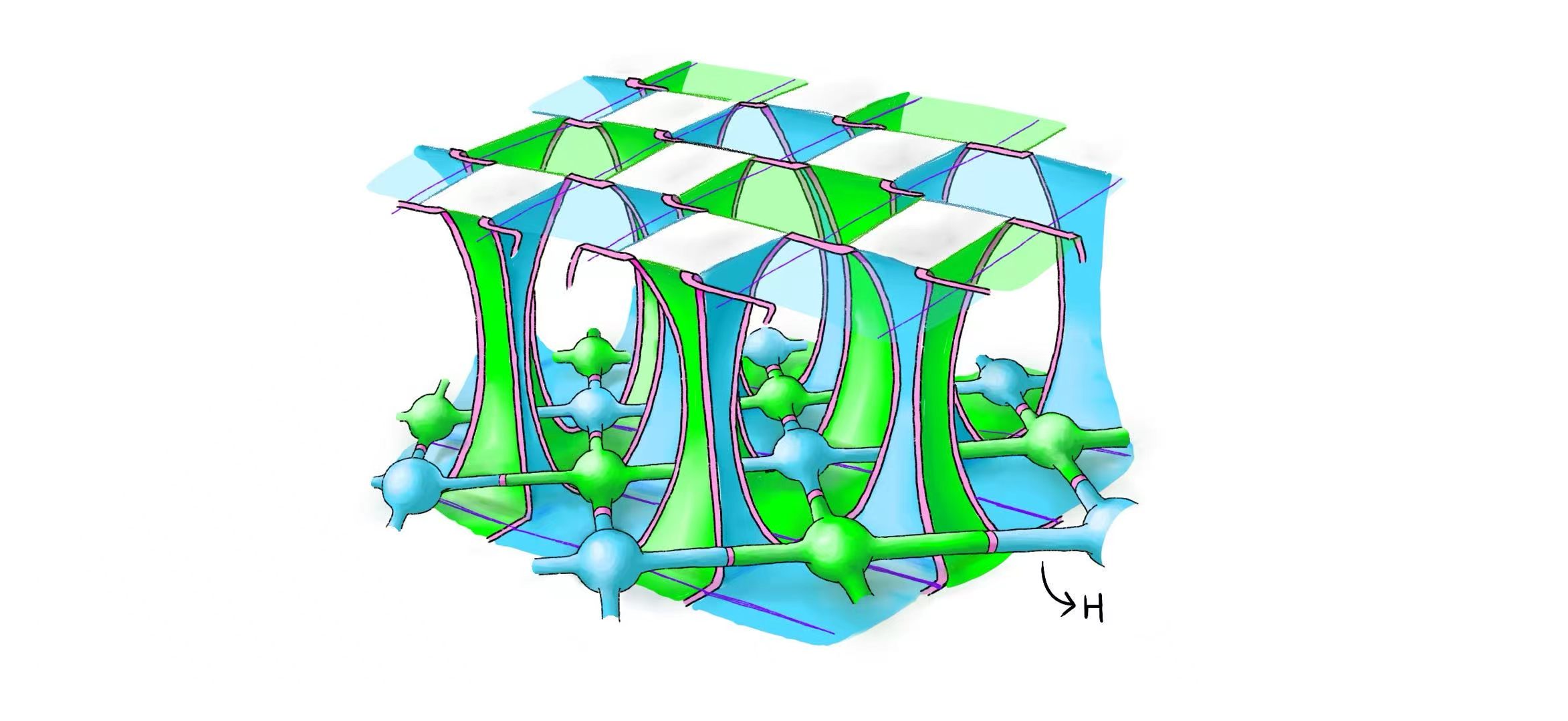}
    \caption{Filling a meridional sutured handlebody}
    \label{fig_3_4_step_3}
\end{figure}

The resulting sutured manifold $(H, \gamma_H)$ is depicted in Figure \ref{fig_3_4_step_3}; the diagram shown is completed to a quadrilateral, and all other cases are analogous. We first prove the following claim.

\begin{claim}
    $(H, \gamma_H)$ is a meridional sutured handlebody.
\end{claim}

\begin{proof}
    By construction, $\gamma_H$ is a union of annuli. We first show that each suture of $s(\gamma_H)$ bounds a disk in $H$. See Figure~\ref{fig_3_5_claim} for an illustration of the construction below. Each suture component $s_i \in s(\gamma_X)$ has nonempty intersection with the set of 4-valent vertices of the diagram $D_0'$, i.e., $(\alpha'\cap\beta')\times \{0,1\}$. Without loss of generality, assume $p_i\times \{0\} \in s_i$. It follows  from the construction in Step 2 that $p_i\times \{1\} \in s_i$ as well, and
    \[
    s_i\cap((\alpha'\cap\beta')\times \{0,1\}) = \{p_i\times \{0\}, p_i\times \{1\}\}.
    \]
    Thus $s_i$ can be divided into two subarcs  $s_i = s_i' \cup s_i''$ such that $\partial s_i' = \partial s_i'' = \{p_i\times \{0\}, p_i\times \{1\}\}$. From Step 2, both $s_i' \cup (p_i\times I)$ and $s_i'' \cup (p_i\times I)$ bound disks in $M\setminus \text{Int}{X}$, and these two disks are disjoint. Hence $s_i$ bounds a disk in $M\setminus \text{Int}{X}$. By the definition of $(H, \gamma_H)$, every component $s_i \in s(\gamma_H)$ bounds a disk in $H$.

\begin{figure}[H]
    \centering
    \includegraphics[width=1\linewidth]{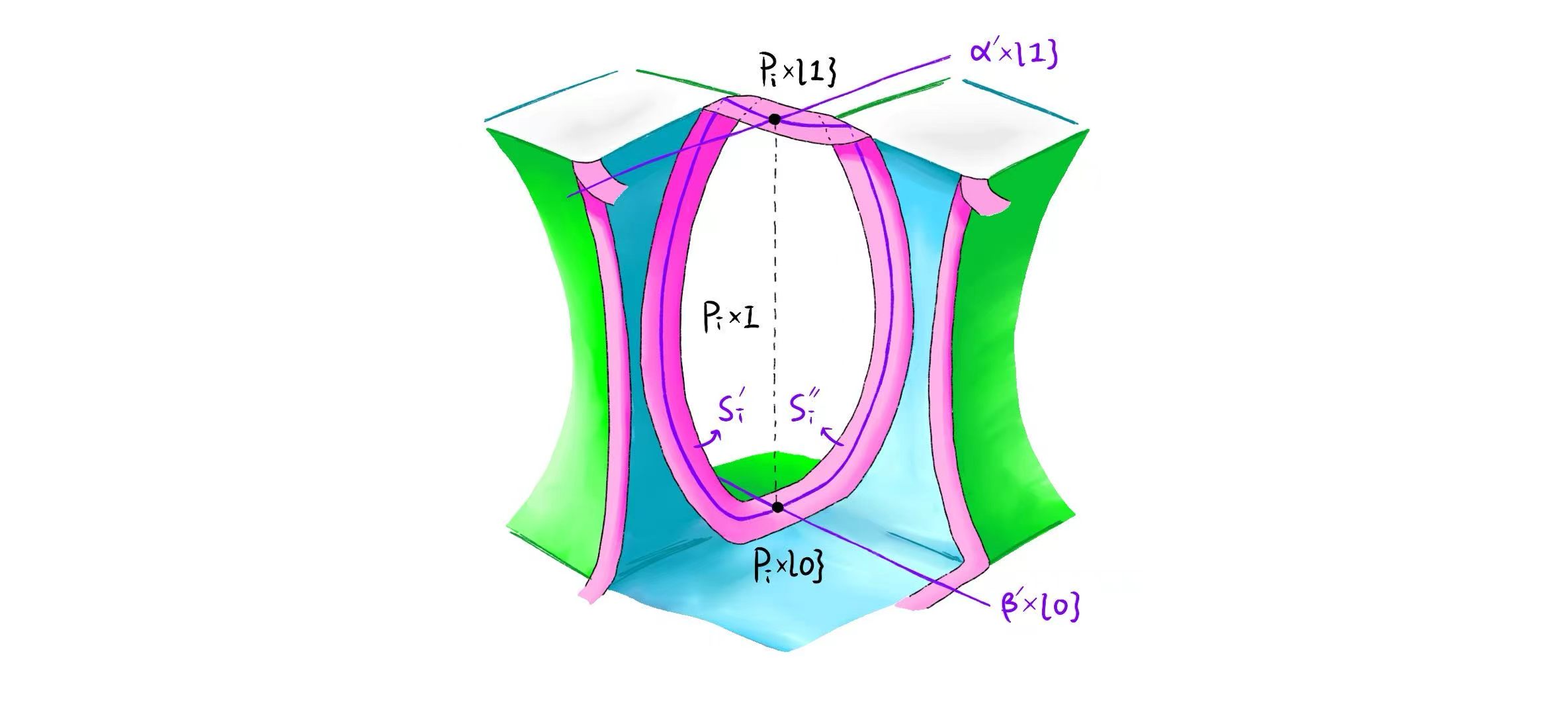}
    \caption{Each suture in $s(\gamma_X)$ bounds a disk in $M\setminus \text{Int}{X}$}
    \label{fig_3_5_claim}
\end{figure}

    It follows from the above argument that the suture set $s(\gamma_H)$ of $(H, \gamma_H)$ is in one-to-one correspondence with the 4-valent vertices of the diagram $D_0'$, as each 4-valent vertex is traversed by exactly one suture. Moreover, each component of $R(\gamma_H)$ is a planar surface formed by gluing the upper and lower copies of the Type I disk components. As shown in Figure \ref{fig_3_4_step_3}, cutting $H$ along the disks corresponding to $s(\gamma_H)$ yields a collection of 3-balls, which are in one-to-one correspondence with the Type I disks.

    Therefore the structure graph $\Gamma(H, D(\gamma_H))$ has vertices corresponding to the Type I disks of $D_0'$ and edges corresponding to the 4-valent vertices of $D_0'$. Every vertex of the structure graph is incident to at least two edges, since every polygon bounded by the diagram has an even number of edges. So, the structure graph $\Gamma(H, D(\gamma_H))$ contains no isolated vertices. Since $\chi(R_+(\gamma_X)) = \chi(R_-(\gamma_X))$ by the previous construction, the equality $\chi(R_+(\gamma_H)) = \chi(R_-(\gamma_H))$ holds immediately.
\end{proof}

By Proposition \ref{prop:msh-foliation}, $(H,\gamma_H)$ admits a foliation. We define \(\hat{X} = (X \cup H) \cup (R(\gamma_X) \times I)\).  Since \(R(\gamma_X) \times I\) admits a foliation of a product I-bundle,  with foliations on \(X\) and \(H\) constructed as above, it follows immediately that \(\hat{X}\) admits a foliation that is transverse to its boundary \(\partial \hat{X}\), which consists of toroidal components. 

Since \(M = \hat{X} \cup_{\partial \hat{X}} (\gamma_X \times I)\),  we apply the spinning construction to the transversely intersecting boundary components, and then equip each component of $\gamma_X \times I$ with a Reeb component. This gives a foliation on \(M\). 


\end{document}